\documentclass[11pt]{article}
\usepackage[margin=1in]{geometry}
\usepackage[utf8]{inputenc} 
\usepackage[T1]{fontenc}   
\usepackage{amsmath}
\usepackage{mathtools}
\usepackage{lipsum}
\usepackage{amsfonts}
\usepackage{graphicx}
\usepackage{epstopdf}
\usepackage{algorithmic,url}
\usepackage{amsthm}
\usepackage[numbers]{natbib}

\newtheorem{prop}{Proposition}

\newtheorem{lem}{Lemma}

\newtheorem{rem}{Remark}

\newcommand{\RR}{{\mathbb R}}

\newcommand{\norm}[1]{\lVert#1\rVert}

\newcommand{\tr}{{\mathrm{Tr}}}
\newcommand{\Sy}{\mathrm{S^+(p,d)}}
\newcommand{\Si}{\mathrm{S_{\mathrm{diag}}^+(p,d)}}
\newcommand{\Sii}{\mathrm{S_{\mathrm{isot}}^+(p,d)}}
\newcommand{\St}{\mathrm{V(p,d)}}

\newcommand{\dotex}{{\frac{d}{dt}}}

\begin{document}

\title{Low-rank plus diagonal approximations for Riccati-like matrix differential equations}

\author{%
  Silvère Bonnabel \\
  MINES ParisTech, PSL University, Center for robotics\\
  \texttt{silvere.bonnabel@mines-paristech.fr}  
  \and
  Marc Lambert  \\
    INRIA - Ecole Normale Sup\'erieure - PSL Research university  \\
  DGA/CATOD, Centre d'Analyse Technico-Op\'erationelle de D\'efense\\
  \texttt{marc.lambert@inria.fr} 
  \and
  Francis Bach \\
  INRIA - Ecole Normale Sup\'erieure - PSL Research university  \\
  \texttt{francis.bach@inria.fr} 
}

%\date{\today}
 \date{}
 
 \maketitle

% REQUIRED
\begin{abstract}
We consider the problem of computing tractable approximations of time-dependent $d\times d$ large positive semi-definite (PSD) matrices    
defined as solutions of a matrix differential equation.  We propose to use  ``low-rank plus diagonal" PSD matrices as approximations   that can be stored with a memory cost being linear in the high dimension $d$. To constrain the solution of the differential equation to remain in that subset, we project the derivative at all times  onto the tangent space to the subset,  following the methodology of dynamical low-rank approximation. We derive a closed-form formula for the projection, and show that after some manipulations it can be computed with a numerical cost being linear in $d$, allowing for tractable implementation.    Contrary to previous approaches based on pure low-rank approximations, the addition of the diagonal term allows for  our approximations to be invertible matrices, that can moreover be inverted   with linear cost in $d$. We apply the technique to Riccati-like equations, then to two particular problems. Firstly a low-rank approximation to our recent Wasserstein gradient flow  for  Gaussian approximation of posterior distributions in approximate Bayesian inference, and secondly a novel low-rank approximation of the Kalman filter  for high-dimensional systems. Numerical simulations illustrate the results.  
\end{abstract}

\section{Introduction}

 Positive semi-definite (PSD) matrices $X \in\RR^{d\times d}$    have   storage cost in   $d^2$, and the computation cost associated with typical matrix operations is in $d^3$. Those two aspects  may come as a limitation in various applications where the dimension $d$ is  very large.  A sensible approach is thus to work with low-rank approximations instead. A rank $p$ approximation
of a PSD matrix $X$ can be factored as $Y=URU^T$ where   $U \in \RR^{d \times p}$ has orthonormal columns and $R$ is a $p\times p$ positive definite matrix of   much-reduced size, letting $p \ll d$.  We denote by $\Sy$ the set of  rank-$p$ PSD matrices.  A somewhat richer set, that will be the object of this paper, is the set of ``low-rank plus diagonal'' matrices, that we denote by $\Si$, composed of matrices of the form $Y=URU^T+\psi$ with $\psi$ a diagonal matrix with strictly positive diagonal elements, making $Y$ invertible. We then  speak of  FA decomposition, in relation to the problem of factor analysis (FA) \cite{Harman}. In the   case where the diagonal matrix $\psi$ is made isotropic, that is, is taken of the form $\psi=s I_d$, $s>0$, we speak of PPCA decomposition, in relation to the problem of probabilistic principal component analysis (PPCA) \cite{Tipping}, and denote this smaller set by $\Sii$. 
{A well-known ``static" approximation problem consists in computing the closest approximating matrix $Y$ to some large PSD matrix $X$ in the sense of the Frobenius norm, that is, solving   
\begin{align}\min_{  Y\in \mathcal  M}||  Y-  X||,\label{Lubich_0:eq}\end{align}
where $\mathcal M$ may denote  $\Sy,\Si$, or $\Sii$. If  $\mathcal M=\Sy$, the problem is  solved by a singular value decomposition   where eigenvalues are truncated after the $p$-th one, a classical result of matrix analysis known as the  Eckart–Young–Mirsky theorem \cite{eckart1936approximation}, see also \cite{Higham}. If $\mathcal M=\Si$, we recover the problem known as ``minimum residual factor analysis", see \cite{shapiro1982rank,bertsimas2017certifiably}. The problem is not solvable in closed form, but numerical methods have been developped, see e.g. \cite{ciccone2019alternating}. }

In this paper, however,  we consider    a different problem. We consider positive semi-definite   matrices $X(t)\in\RR^{d\times d}$ defined as solutions of matrix differential equations in high dimension. Instead of  attacking \eqref{Lubich_0:eq} directly,   we   try to solve at all times 
\begin{align}\min_{  \dot Y(t)\in \mathcal T_{Y(t)}\mathcal  M}||  \dot Y(t)-  \dot X(t)||.\label{Lubich_22:eq}\end{align}{Technically, the problem differs from \eqref{Lubich_0:eq} as the derivatives live in a different space, namely the tangent spaces. The rationale is as follows.  To handle possibly high dimension $d$, we would like to maintain a storage cost being linear in $d$--which is achieved by letting  $Y(t)\in\mathcal M$ at all times--and a computation cost being also linear in $d$ when incrementally evolving our approximation $Y(t)$, hence the need to solve \eqref{Lubich_22:eq} efficiently.  The solution to  
\begin{align}
\dot X(t)=F(X(t))
\label{EDO:eq}\end{align} may then be approximated by letting at all times $\dot Y(t)$ be the solution to \eqref{Lubich_22:eq}, where $\dot X(t)=F(X(t))$  is replaced by $F(Y(t))$. Indeed, even if $Y(t)\in\mathcal M$,  the matrix $F(Y(t))$ generally points to a direction that makes $Y(t)$ step out of $\mathcal M$, hence the need to project it onto the tangent space $\mathcal T_{Y(t)}\mathcal  M$. 
 Although   $\mathcal T_{Y(t)}\mathcal  M$ is a vector space for each $Y(t)$, solving \eqref{Lubich_22:eq} exactly while maintaining operations being linear in $d$ is not straightforward, and this is the object of the present paper. }

Let us first consider $\mathcal M=\Sy$. Low-rank approximations $Y(t)$ to $X(t)$  conveniently write  $Y(t)=U(t)R(t)U(t)^T$ with $U(t)\in\RR^{d\times p}$ a matrix whose columns are orthonormal and  $R(t)\in\mathrm{S^+(p,p)}$  a small-size positive definite (PD) symmetric matrix, as advocated in \cite{Lubich, Bonnabel, Rouchon}.
  \cite{Lubich, Rouchon}  address the following projection problem
\begin{align}\min_{\dot Y(t)\in \mathcal T_{Y(t)}\Sy}||\dot Y(t)-\dot X(t)||\label{Lubich:eq}\end{align}in the sense of the Frobenius norm, where $\mathcal T_{Y}\Sy$ denotes the tangent space to $\Sy$ at $Y\in\Sy$.  \cite{Lubich} derives expressions for $\dot U(t)$ and $\dot R(t)$ such that $\dot Y(t)$   corresponds to the solution of \eqref{Lubich:eq} indeed.

Historically, \cite{bonnabel2013geometry,bris2012low,Rouchon} proposed  approximations to the  continuous-time low-rank Riccati  equation using differential  geometry, which were in particular applied to approximating the Linblad equations in quantum physics. The work \cite{bris2012low,Rouchon} is in line with prior results of \cite{Lubich}, where essentially one seeks to orthogonally project the tangent vector in the sense of the Euclidean norm, whereas the approach of \cite{bonnabel2013geometry} builds on the geometry of \cite{Bonnabel}. {A judicious numerical time integrator was proposed in \cite{lubich2014projector} with recent extensions to matrices and tensors that preserve symmetry  in   \cite{ceruti2020time}. }

The present paper builds on  \cite{Lubich, bris2012low,Rouchon},  providing an invertible covariance matrix, by approximating derivatives on $\Si$ and $\Sii$ instead. This has  not been addressed to our knowledge, and we provide closed-form formulas for the projection of the tangent vector. Although the method is   general, implementing the projection depends on the underlying differential equation being approximated. Of particular interest is the Riccati differential equation, that appears in a number of applications, e.g., Kalman filtering, linear quadratic control, in stability analysis via the Lyapunov function, or in computational statistics to estimate covariance matrices, see \cite{lambert2022recursive}. The interests of the ``low-rank plus diagonal" framework {in the context of matrix dynamical approximation} are as follows: \begin{itemize}\item It allows for manipulation of \emph{invertible} matrices, whereas low-rank approximations are inherently non-invertible. Moreover, using Woodbury's lemma, inversion of such an approximation is feasible at all times,  see \eqref{woodbury:eq} below.
 \item  Approximation with a full rank matrix is desirable in estimation problems where  $Y$ encodes a covariance matrix. A rank-deficient covariance may lead to over-confident estimates.  In Kalman filtering, this may lead the  state variable estimates to spuriously drift, as shall be shown in simulations. \item It ensures more flexibility while retaining a  storage cost  linear in $d$. Notably, it can   capture the individual variances of the variables for a large  covariance matrix $Y$. 
 \item $Y$ admits a   probabilistic interpretation: $URU^T+sI_d$ is the covariance of $x=Uz+\nu$, with $z\sim\mathcal{N}(0,R)$ a small-dimensional latent variable  mapped via $U$  into a   subspace of $\RR^d$, plus $\nu$ an isotropic noise of magnitude $ s,$ \cite{Tipping,Harman}.
 \end{itemize}

Our contributions and the organization of the paper are as follows. Previous work on low-rank approximation {of matrix differential equations} \cite{Lubich, Rouchon} is recapped in Section \ref{sec:0}. Section \ref{sec:1} extends the results to the low-rank plus diagonal case: We solve  exactly in closed form for each $t$ the following extension of problem  \eqref{Lubich:eq}:\begin{align}\min_{\dot Y(t)\in \mathcal T_{Y(t)}\mathcal \Si }||\dot Y(t)-\dot X(t)||,\label{Lubich2:eq}  \end{align} for the set $\Si$ and also for $\Sii$. In Section  \ref{sec:2}, we discuss  implementation in high dimension and show the approximation can be obtained exactly with a computational cost being linear in $d$.  In Section  \ref{Riccati:sec}, we leverage the result to provide a tractable approximation to the Riccati equation. This is applied to two problems. First, we derive a low-rank approximation to our Wasserstein gradient flow of \cite{lambertNIPS}. Then, we propose a novel low-rank plus diagonal Kalman filter and illustrate the results on a tutorial example inspired by robotics.   The code is made publicly available.

 %Factor analysis (FA) \cite{Harman} is a standard technique of statistics to identify   latent factors explaining correlated observations. Its extension,  the probabilistic principal component analysis 
 %(PPCA) \cite{Tipping}, has been proposed to endow PCA in data analysis with a probabilistic model. This approach has been extended in the context of Bayesian online learning in \cite{Lambert23}. 
 As concerns low-rank techniques for Kalman filtering, we note the problem  has arisen in data assimilation for weather forecasting and oceanography \cite{Evensen94}, \cite{Pham}, where the state is driven by a partial differential equation, and is thus encoded by a high dimensional vector (say, up to 1 million). Since the Kalman filter needs to store the covariance matrix of the estimates, it meets the computer's memory limits when the dimension is high.  Different variants have been proposed to tackle this memory problem. The SEEK filter \cite{Pham} is based on an SVD decomposition of the covariance matrix and provides a low-rank Riccati update in discrete time. The ensemble Kalman filter \cite{Evensen94,le2022low} approximates the Riccati equation using  Monte Carlo  sampling.  {Very recently, and closely related to the present paper,  a method for efficient, approximate Gaussian (Kalman) filtering, including
smoothing and marginal likelihood computation was developed using low-rank approximations on $\Sy$, see \cite{schmidt2023rank}. } In robotics, \cite{Thrun}  proposed the sparse information filter (SEIF)  filter to process large maps for simultaneous localization and mapping.

\section{Reminders on low-rank  approximation}\label{sec:0}

We start with the geometry of $\Sy$, and then recap existing results on dynamical low-rank approximation. 
\subsection{Geometry of $\Sy$}
 Any matrix $Y\in\Sy$ may be written as $Y=URU^T$ %, with $U\in \St$, 
 where $R\in \mathrm {S^+(p,p)}$ is a small size positive definite matrix, and  $U\in \St$ where $\St$ denotes the set of $d\times p$ matrices with $p$ orthonormal columns, hence satisfying $U^TU=I_p$, called  the Stiefel manifold. Letting  $\Pi_U=UU^T$ be the projector onto the span of $U$ and $\Pi_U^\perp=I_d-UU^T$ the projector onto the orthogonal subspace,  we have $
 (\Pi_U)^2=\Pi_U,~(\Pi_U^\perp)^2=\Pi_U^\perp,~\Pi_U^\perp U=0,~U^T\Pi_U^\perp=0.
 $

 The quotient geometry of $\Sy$ is thoroughly studied in \cite{Bonnabel}.  Any tangent vector to $ \Sy$ may be represented by the infinitesimal variation $(\delta U,\delta R)    \in \mathcal T_{(U,R)}\Sy$  with $(\delta U,\delta R)  \in \RR^{d\times p}\times \RR^{p\times p}$ of the following form
\begin{align}
 \delta U=(I_d-UU^T)\Gamma,~\Gamma\in\RR^{d\times p};~ \delta R \in\RR^{p\times p},~\delta R^T=\delta R.\label{tan:form:eq}
\end{align}
This may be interpreted as follows: An infinitesimal variation of $U$ makes  its  columns move in the subspace of $\RR^d$ orthogonal to $\mathrm{span}(U)$, {as we have $(\delta U)^\top U = 0$, ensuring that $\delta (U^TU)=0$, in accordnace with the constraint $U^TU=I_p$}. And an infinitesimal variation of $R$  necessarily remains symmetric~\cite{Bonnabel}. The corresponding tangent vector $\delta Y=\mathcal T_Y\Sy$  at $Y=URU^T$ writes
 \begin{align}
 \delta  Y=(\delta U )RU^T+U (\delta R) U^T+UR(\delta U)^T.\label{tangentLR}
 \end{align}

\subsection{Optimal approximations on the tangent space}

 Any tangent vector to the set of full $d\times d$ PSD matrices at $X$ is encoded by a matrix $H\in\RR^{d\times d}$, with $H$ symmetric, see, e.g., \cite{Bonnabel}. Hence, given matrices $X(t)$ depending smoothly on the parameter $t\in\RR$, one may write at each time $\dot X(t)=H$, {and turn the problem   \eqref{Lubich_22:eq}, which specifies to \eqref{Lubich:eq} 
 in the present case, into the generic problem  of solving 
 }\begin{align}\min_{\dot Y(t)\in \mathcal T_{Y(t)}\mathcal \Sy }||\dot Y(t)-H||,\label{Lubich3:eq}  \end{align} with $H$ symmetric and arbitrary. Replacing $H$ with $F(Y(t))$ at each time $t$, and completing it with an initial condition $Y(t_0)$ will yield  an approximation $ \big(Y(t)\big)_{t\geq t_0}$ to the solution of the ODE $\dot X=F(X)$.   
 We have the following:  \begin{prop}[from \cite{Lubich,Rouchon}]\label{Lubich:prop}
 The orthogonal projection  of a symmetric matrix $H$ onto $\mathcal T_Y \Sy$ at  $Y=URU^T$ is,  in the retained form of tangent vectors \eqref{tan:form:eq},
$$
 \delta  Y=P_{U,R,s}(H)= 
\delta U RU^T+U \delta RU^T+UR\delta U^T,
$$
where the matrices are given by:
\begin{align}
%\Delta(H,U,R) &=(I-UU^T)  HU R^{-1},\label{Delta}\\
\delta U &=(I-UU^T)  HU R^{-1},\label{Delta}\\
%D (H,U)&=U^T HU\label{D}.
\delta R &=U^T HU\label{D}.
\end{align}
The tangent vector then writes$$
P_{U,R,s}(H)=HUU^T +UU^TH-UU^THUU^T.
$$
This choice   solves   problem \eqref{Lubich3:eq}, that is, it minimizes over matrices of the form $ \delta  Y= 
\delta U RU^T+U \delta RU^T+UR\delta U^T$ with constraints \eqref{tan:form:eq} the following cost
\begin{align}J_{H;U,R}( \delta U, \delta R)=\tr((  H-\delta Y)^2)=||  H-\delta Y||^2,\label{cost}
\end{align}whose minimum is given by $\min J_{H;U,R}( \delta U, \delta R)=\tr((I-UU^T)H(I-UU^T)H)$, and where we used { the trace form of Frobenius norm $||A||^2=\tr(A^TA)=\tr(A^2)$ for symmetric matrices}.
 \end{prop}
Given a differential equation $\dotex X(t)=F(X(t))$ over PSD matrices of full size    $d\times d$, Equations \eqref{Delta}-\eqref{D} readily yield $\dot U(t)$ and $\dot R(t)$ to implement its low-rank approximation at all times,   letting $H=F(Y(t))$ with $Y(t)=U(t)R(t)U(t)^T$. 

\section{Optimal low-rank plus diagonal approximation}\label{sec:1}
We now turn to our main theoretical results  which consist of   extensions to the low-rank plus diagonal case.
\subsection{Geometry of $\Si$ and $\Sii$ } 

%In factor analysis \cite{Harman}, we construct a matrix of the form $Y=URU^T+\psi$ with $\psi$ diagonal. This leads to a manifold we denote $\Si$. The PPCA representation \cite{Tipping} is a particular case where   we restrict the diagonal matrix to be of the form $\psi=sI_d$, leading to the manifold $\Sii$. 

Take  $Y=URU^T+\psi$ with $\psi$ diagonal and having all its diagonal elements strictly positive. Recalling \eqref{tan:form:eq}, an  element $Y\in\Si$, along with tangent vectors $\delta Y\in \mathcal T_Y \Si$ to this element,  write \begin{equation} \begin{rcases}
Y&=URU^T+\psi\\\delta Y&=\delta U RU^T+U \delta R U^T+UR\delta U^T+\delta \psi
\end{rcases}\quad\text{(FA)}\label{tan:psi}
\quad \end{equation}
with  $\delta U,\delta R$ as in \eqref{tan:form:eq} and  $\delta \psi$ diagonal. 
\begin{rem}

    {In terms of geometric structure, $\Sy$ is a smooth embedded submanifold, albeit not straightforward to prove. The interested reader is referred to \cite{vandereycken2010riemannian}, which contains two proofs: Prop. 3.11, and another from \cite{helmke1992critical} which is completed therein. Unfortunately, if we wonder whether this extends to $\Si$, we discover  that $\Si$ is not a submanifold: The dimension of the tangent space is not constant, as can be seen by considering $Y=xx^T+\psi$ with $x=(1,0,0)^T$ on the one hand, leading to a dimension of 5 for all $\delta Y$,   and $x=(1,1,1)^T$  on the other,   then leading to a dimension of 6.  Thus, ${\mathrm{S_{\mathrm{diag}}^+(1,3)}}$ bears no manifold structure, but this shall  not prevent one from using \eqref{tan:psi} to represent the   matrices and their derivatives, and refer to the set of all such $\delta Y$ as the tangent space at $Y$. } 
    
\end{rem}

Let us now turn to the PPCA decomposition, of the form $Y=URU^T+sI_d$. {So far, it has been introduced as a subset of $\Si$ based on the constraint $\varphi=sI_d$, for the sake of simplicity of exposition. However, because some degrees of freedom are removed by this constraint,  the eigenvalues of $Y$ are  necessarily lower bounded by $s>0$, as $Y\succeq sI_d$, since $URU^T$ is PSD.   Thus, $Y$ cannot have arbitrarily small eigenvalues in $\mathrm{span}(U)$. This motivates the use of a slightly different--richer--parameterization. From now onwards, we  define $\Sii$ as matrices of the form 
\begin{align}Y=U(R-sI)U^T+sI=URU^T+s(I-UU^T)=URU^T+s\Pi_U^\perp,\end{align}and impose that $R\in \mathrm{S^+(p,p)}$, $U\in \St$ and $s>0$. This allows for $\Sii$ to be a bigger approximating set, while retaining that $Y$ is positive definite. }

 This leads to the following geometry for $\Sii$. An element $Y\in\Sii$ with associated tangent vector $\delta Y\in \mathcal T \Sii$ in the PPCA form writes\begin{equation} \begin{rcases}
Y&\!\!\!\!=URU^T+s(I-UU^T)=U(R-sI)U^T+sI\\\delta Y&\!\!\!\!=\delta U (R-sI)U^T+U(\delta R-\delta s I)U^T\\&\hspace{3cm}+U(R-sI)\delta U^T+\delta s I
\end{rcases} \text{(PPCA)}\label{tan:psi2}
\quad \end{equation}
with  $\delta U,\delta R$ as in \eqref{tan:form:eq}    and $\delta s\in\RR$.

\subsection{Optimal approximation  in the PPCA form}

Let us start with the simpler PPCA decomposition.
\begin{prop}\label{prop:PPCA}
The orthogonal projection   of a symmetric matrix $H\in\RR^{d\times d}$ onto   $\mathcal T_Y\in\Sii$   is
$\delta Y =\delta U (R-sI)U^T+U(\delta R-\delta s I)U^T +U(R-sI)\delta U^T+\delta s I$ where the matrices are given by: 
\begin{equation}
\begin{aligned}
\delta s&=[\tr(H)-\tr(U^THU)]/(d-p),\\
\delta U &=(I-UU^T)  HU( R-sI)^{-1},\\
\delta R &=U^T HU.
\end{aligned}\label{matrices:PPCA}\end{equation}
The tangent vector then writes
$$P_{U,R,s}(H)=HUU^T +UU^TH-UU^THUU^T+
 \delta s (I-UU^T),$$
This choice minimizes over matrices of the form \eqref{tan:psi2} the following cost
\begin{equation} C_{H;U,R,s}(\delta U, \delta R, \delta s)=\tr((  H-\delta Y)^2).\label{cost42}\end{equation}
\end{prop}
\begin{proof}
The cost \eqref{cost42}, that is, the squared Frobenius norm between $H$ and the tangent vector,  rewrites
$$\tr((\tilde H-[\delta U \tilde RU^T+U\delta R U^T+U\tilde R\delta U^T])^2),
$$with $\tilde H=H-\delta s (I-UU^T)$ and $\tilde R=R-sI$. In other terms
$$
C_{H;U,R,s}(\delta U, \delta R, \delta s)=J_{\tilde H;U,\tilde R}(\delta U,\delta R)  ,      
$$with $J$ as in  \eqref{cost}. The next step consists in writing
\begin{align}
\min_{\delta U, \delta R, \delta s}C_{H;U,R,s}&=\min_{\delta s} \bigl(\min_{\delta U, \delta R}C_{H;U,R,s}\bigr)\\&=\min_{\delta s} \bigl(\min_{\delta U, \delta R}J_{\tilde H;U,\tilde R} \bigr)
\\&=\min_{\delta s}\tr{\bigl((I-UU^T)\tilde H(I-UU^T)\tilde H \bigr)},\label{lasteq}
\end{align}using Proposition  \ref{Lubich:prop}, which also shows  that  in the present case $\min_{\delta U, \delta R}J_{\tilde H;U,\tilde R}$ corresponds to
\begin{align}
\delta U &=(I-UU^T)\tilde HU\tilde R^{-1}=(I-UU^T)  HU\tilde R^{-1}\\
\delta R &=U^T\tilde HU=U^T HU.
\end{align} 
 
{To conclude, we still have to solve \eqref{lasteq}}. Recalling $\tilde H=H-\delta s (I-UU^T)$ we  differentiate $\tr{\bigl((I-UU^T)\tilde H(I-UU^T)\tilde H \bigr)}$ with respect to $\delta s$  
\begin{align*}
J'( \delta s)&=-2\tr{\bigl((I-UU^T)\tilde H\bigr)} \\&=-2\tr{\bigl((I-UU^T)(H-\delta s (I-UU^T))\bigr)}
\end{align*}
and we find  $\delta s=[\tr(H)-\tr(U^THU)]/(d-p)$. 
\end{proof}

\subsection{Optimal approximation in the FA form}
The problem of projecting onto $\Si$ is more difficult than the latter, as in the proof of Proposition \ref{prop:PPCA} we extensively used the isotropy of the diagonal term $sI_d$, namely $U(sI_d)U^T=sUU^T$. However, it turns out we may find a closed-form expression also in the FA decomposition.  This allows for   capturing different individual variances in the subspace orthogonal to $\mathrm{span}(U)$.   

The first step is to transform the optimization problem over the  matrix $\psi$ into a least-squares problem for a vector consisting of its diagonal. 
The cost we want to minimize is
\begin{align*}
&C_{H;U,R, \psi}(\delta U, \delta R, \delta \psi)\\&:=\tr((H-[\delta U RU^T+U\delta R U^T+UR\delta U^T+\delta \psi])^2)\\
&=\tr((H-\delta \psi-[\delta U RU^T+U\delta R U^T+UR\delta U^T])^2)\\&=
\tr((\tilde H-[\delta U RU^T+U\delta R U^T+UR\delta U^T])^2)\\&=J_{ \tilde H, U,R}(\delta U, \delta R),
\end{align*}where now $\tilde H$ denotes  $\tilde H=H-\delta \psi$.   We write $$
\min_{\delta U, \delta R, \delta \psi}C(\delta U, \delta R, \delta \psi)=\min_{\delta \psi} \bigl(\min_{\delta U, \delta R}J_{\tilde H;U,\tilde R}(\delta U,\delta R) \bigr),
$$which is equal to $\min_{\delta \psi} J_{\tilde H;U,  R}(\delta U,\delta R) $ with
\begin{align}
\delta U &=(I-UU^T)\tilde HUR^{-1},\label{UUeq1}\\
\delta R &=U^T\tilde HU,\label{UUeq2}
\end{align}
 using \eqref{Delta}, \eqref{D}. Using Proposition \ref{Lubich:prop}, we see the cost function to minimize becomes
\begin{equation}
\begin{aligned}
\tilde C(\delta \psi)&:=\tr\bigl((I-UU^T)\tilde H(I-UU^T)\tilde H\bigr)\\&=\tr\bigl(\tilde H^2\bigr)-2\tr\bigl(UU^T\tilde H^2\bigr)+\tr\bigl(UU^T\tilde H UU^T\tilde H\bigr).\label{cost0}
\end{aligned}\end{equation} 
\eqref{cost0} may be re-written as a vector least squares problem for the vector $\overline{\delta \psi}:=(\delta \psi_{11},\delta \psi_{22},\dots,\delta \psi_{dd})^T={\rm diag}(\delta \psi)\in\RR^d$. This leads to the following optimal approximation.

\begin{prop}\label{prop:FA}The orthogonal projection  of a symmetric matrix $H\in\RR^{d\times d}$ onto $\mathcal T_Y\in\Si$ is $\delta Y=\delta U RU^T+U \delta R U^T+UR\delta U^T+\delta \psi$ where the matrices  are given by\begin{equation}
\begin{aligned}&{\rm diag}(\delta \psi)  =\overline{\delta \psi} , \\
&\delta U  =(I-UU^T) (H-\delta \psi )UR^{-1}, \\
&\delta R =U^T(H-\delta \psi )U,
\end{aligned} \label{matrices:FA3}   
\end{equation}and where the vector  $\overline{\delta \psi} \in\RR^d$  that encodes the diagonal matrix $\delta\psi\in \RR^{d\times d}$ is defined by\begin{align}
 \overline{\delta \psi} &=((I-UU^T)^{\circ 2})^{+}{\rm diag}((I-UU^T)H (I-UU^T)) , \label{matrices:FA1}\end{align}letting $+$ denote the Moore-Penrose inverse and where $\circ$ denotes the   element-wise  (Hadamard) matrix product. 
 \end{prop}
  \begin{proof}
      Recall the  notation  $\Pi_U^\perp=I_d-UU^T$. We may rewrite the cost \eqref{cost0} in terms of the vector $\overline{\delta \psi}$, i.e., $\tilde C(\delta \psi)=\bar C(\overline{\delta \psi})$ and we find:
\begin{align*}
&\bar C(\overline{\delta \psi})={\rm Tr}\bigl((\Pi_U^\perp H-\Pi_U^\perp{\rm diag}(\overline{\delta \psi}))^2\bigr)\\
&=c-2{\rm Tr}(\Pi_U^\perp H \Pi_U^\perp{\rm diag}(\overline{\delta \psi}))+\rm{Tr} (diag(\overline{\delta \psi}) \Pi_U^\perp {\rm diag}(\overline{\delta \psi})\Pi_U^\perp),
\end{align*} 
where $c$ is a constant with respect to $\overline{\delta \psi}$. This may be re-written as a vector least-squares problem for the vector $\overline{\delta \psi}$ using the relation ${\rm Tr} ({\rm diag}(x)A{\rm diag}(y)B^T)=x^T(A \circ B)y$ where $\circ$ is the element-wise matrix product. If we note $\mathrm{1_d}$ the vector of ones, the cost becomes:
$$\tilde C(\delta \psi)=c-2 \mathrm{1_d}^T (\Pi_U^\perp H \Pi_U^\perp \circ I) \overline{\delta \psi} +\overline{\delta \psi}^T(\Pi_U^\perp \circ \Pi_U^\perp)\overline{\delta \psi}.$$
Taking the derivative with respect to  $\overline{\delta \psi}$, we obtain: $$\nabla \bar C (\overline{\delta \psi})=-2{\rm diag}(\Pi_U^\perp H \Pi_U^\perp)+2(\Pi_U^\perp \circ \Pi_U^\perp)\overline{\delta \psi}$$Zeroing the gradient yields \eqref{matrices:FA1} indeed.
  \end{proof}

 \section{Implementation}\label{sec:2}
 
{ We now discuss the interest of the results for the high-dimensional setting. Our goal, for tractability,  is to elicit  operations and storage costs being linear in the dimension $d$.}  We prove that both methods lead to projections that require linear computation cost in the dimension $d$, and discuss the numerical cost in detail. A key result to our analysis is that any matrix $Y$  of $\Si$ or $\Sii$ is easily inverted via the Woodbury  lemma: 
 \begin{align}(URU^T+\psi)^{-1}=\psi^{-1}-\psi^{-1}U(R^{-1}+U^T\psi^{-1}U)^{-1}U^T\psi^{-1},\label{woodbury:eq}\end{align}
 a fact we will extensively use. Note that it also proves  matrices of $\Si$ and $\Sii$, contrary to low-rank matrices of $\Sy$, can be inverted, and this with a  numerical cost of inversion  is $O(p^3d)$, hence linear in the   dimension $d$.

\subsection{Numerically efficient formulation for the FA form}
Equation \eqref{matrices:FA1}  necessitates (pseudo)-inversion of a $d\times d$ matrix $(I-UU^T)^{\circ 2}$ which may hinder its use in large dimension. However, it is amenable to a linear computation cost. 

\begin{prop}\label{prop:FA2}The orthogonal projection \eqref{matrices:FA1} can be performed with linear computation cost in $d$. Indeed,     the required vector $\overline{\delta \psi}\in\RR^d$ is given by
\eqref{min1} below, and can be advantageously computed via \eqref{argmin2}. 
 \end{prop}
  
To prove the result we start with a lemma, whose proof is technical and hence postponed to the Appendix.    

\begin{lem}
The cost $\tilde C(\delta \psi)$ may be re-written as a function of the vector $ \overline{\delta \psi}\in\RR^d$ as follows:\begin{equation}
 \begin{aligned}\tilde C(\delta \psi)&=\bar C(\overline{\delta \psi}):=\alpha+(\bar h-\overline{\delta \psi})^T(\bar h-\overline{\delta \psi}) 
 \\&~+4\bar h_U^T\overline{\delta \psi}-2\overline{\delta \psi}^T \bar D\overline{\delta \psi}- 2\Lambda^T\overline{\delta \psi}+\overline{\delta \psi}^T\Upsilon \Upsilon^T\overline{\delta \psi},
\end{aligned}\label{barcost}\end{equation}
 where $\alpha$ is a constant, $\bar h,\bar h_U,\Lambda$ are all vectors of $\RR^d$, and were $\bar D$ is a diagonal $d\times d$ matrix, and $\Upsilon$ is a matrix, all being given below. Hence the gradient writes \begin{align}
 \frac{1}{2}\nabla \bar  C(\overline{\delta \psi})=\bigl(I-2\bar D+\Upsilon \Upsilon^T\bigr)\overline{\delta \psi}-\Lambda+2\bar h_U-\bar h,
\end{align}
 and the optimizer is given by
\begin{align}
 \overline{\delta \psi}=\bigl(I-2\bar D+\Upsilon \Upsilon^T\bigr)^{+} \bigl(\bar h-2\bar h_U+\Lambda\bigr),\label{min1}
\end{align}
The parameters in the equations above are:
\begin{align}
   \bar h&={\rm diag}(H),\qquad  \bar h_U:={\rm diag}(UU^TH) \\ \Lambda&=
(\Lambda_1,\dots,\Lambda_d)^T\in\RR^d\\\Lambda_k&=\sum_{1\leq i,j\leq r} (U^THU)_{ij} U_{ki}U_{kj},\quad 1\leq k\leq d 
 \\ \Gamma_{ij}&=((\Gamma_{ij})_1,\dots,(\Gamma_{ij})_d)^T\in\RR^d\\(\Gamma_{ij})_k&=U_{ki}U_{kj},\quad 1\leq k\leq d ,
\end{align}  with $\Upsilon$   a matrix of size $d\times \frac{p(p+1)}{2}$ whose first  $\frac{p(p-1)}{2}$ columns are given by the vectors $\sqrt{2}\Gamma_{ij}$, $i<j$, and whose last $r$ columns are $\Gamma_{ii}$, and $\bar D$ is the diagonal matrix defined by $\bar D_{ii}=\sum_{j=1}^rU_{ij}^2$. It may be checked \eqref{min1} coincides with \eqref{matrices:FA1}.\label{lemme}
\end{lem}
The problem with \eqref{min1} is that it requires inverting a $d\times d$ matrix, which is not affordable computationally.  When $I-2\bar D+\Upsilon \Upsilon^T$ is invertible, we can use Woodbury matrix identity instead to  express  the solution  \eqref{min1}  as 
\begin{align} \overline{\delta \psi}=\bigl(\varphi-\varphi\Upsilon\bigl(I_{\frac{p(p+1)}{2}}+\Upsilon^T\varphi\Upsilon \bigr)^{-1} \Upsilon^T\varphi\bigr)\bigl(\bar h-2\bar h_U+\Lambda\bigr),\label{argmin2}
\end{align}where $\varphi:=(I_d-2\bar D)^{-1}$ is the inverse of a diagonal matrix, which may be efficiently computed using only vectors of $\RR^d$, and where the  $\frac{p(p+1)}{2}\times \frac{p(p+1)}{2}$ matrix $I_{\frac{p(p+1)}{2}}+\Upsilon^T\psi\Upsilon $ needs be inverted.

 \subsection{Computation cost}

As a preliminary remark, we see that whatever the chosen low-rank decomposition, computing the approximation requires storing and projecting the $d\times d$ symmetric matrix $H$. To perform this in a high dimension we need to assume some sort of sparsity, or low-rank structure for $H$ of the kind $H=GG^T$ with $G\in\RR^{d\times r}$ where $r\ll d$, or $H\in\mathrm{S_\psi^+(r,d)}$.   This problem is already present in the literature on low-rank approximations and has proved not to be a limitation in applications, see \cite{Lubich, Rouchon} and articles that followed.  For our complexity analysis, we will henceforth suppose the computational cost of computing $U^THU$ is $O(dp^2)$.  
Note that to project the matrix in the PPCA or FA form, all equations   must be coded ``vectorially'', that is, diagonal matrices are encoded by vectors, and we carefully choose the order of operations to never multiply full-size matrices, e.g.,   in  \eqref{argmin2}. 

 Propositions \ref{prop:PPCA}, \ref{prop:FA} and \ref{prop:FA2} give the formulas that one needs to evolve the factors $U(t),R(t),s(t)$ or $U(t),R(t),\psi(t)$, where their derivatives are given by \eqref{matrices:PPCA} and \eqref{matrices:FA3}, \eqref{argmin2} respectively. 
In the PPCA form, to compute matrices \eqref{matrices:PPCA} we see we need to perform 
 $O(dp^2)$ operations.   
In the FA form,   implementation of \eqref{matrices:FA3}  retains linear complexity in $d$. However in the efficient implementation \eqref{argmin2} of \eqref{matrices:FA1}, we need to invert the  $\frac{p(p+1)}{2}\times \frac{p(p+1)}{2}$ matrix $I_{\frac{p(p+1)}{2}}+\Upsilon^T\psi\Upsilon $. This yields a computational cost of order $O(dp^2)+O(p^6)$. We thus see that FA decomposition--albeit richer--comes at a price:  the number of ``factors'' $p$ that may be used is more limited than in the PPCA decomposition.

 To illustrate  how the methods compare in terms of numerical cost, we propose to apply the three methods for the approximation of a large-scale matrix $H=GG^T$ where $G$ is a matrix of size $d \times r$ generated randomly, where  $r>p$  is low,   but sufficiently large for the $p$-rank approximation not to be   exact. Using a product form for $H$ allows for  low-memory-cost  operations, in particular, we have $UHU^T=(UG)(UG)^T$ yielding a linear cost in $d$. The execution times of the different algorithms are shown in Table \ref{tableXP} in dimension $d=10^6$ where the computations are performed on a standard laptop with Octave. All algorithms scale well to very high-dimensional problems. However, FA has a   higher computation time since it needs to  invert   a $\frac{p(p+1)}{2} \times\frac{p(p+1)}{2}$ matrix.    
\begin{table}[!ht]
\caption{Matrix projection with $d=10^6$, $p=10$, and $r=100$ }
\centering
\begin{tabular}{|c|c|c|c|}
  \hline
 Set & Projection method & Execution time (s)  & Memory cost\\
  \hline
    $\Sy$ &``Low-rank"   & 3.5 s   & $dp$ \\
    $\Sii$ & ``PPCA"  & 7 s    & $dp$\\
    $\Si$ & ``FA"  & 35 s    & $dp+p^4/4$\\
  \hline
\end{tabular}
\label{tableXP}
\end{table}

\section{Application to the Riccati equation}\label{Riccati:sec}

In this section, we consider the following classical continuous-time  Riccati equation in high dimension $d$: 
\begin{align}
\dotex P=AP+PA^T+Q-PC^TN^{-1}CP,\label{Riccati}
\end{align}
with $P$, $A$ and $Q$ are $d\times d$ high dimensional matrices, $N$ is a $k\times k$  matrix and $C$ is of size $k \times d$. $P$, $Q$ and $N$ are  PSD matrices.

To store this covariance matrix $P$ in high dimension a first solution is to project it onto the low-rank manifold $\Sy$. %Let us approximate it in $\Sy$. 
To do so we replace $P$ with $Y=URU^T$ in the right-hand side \eqref{Riccati}, and we project the obtained matrix.  Prop.  \ref{Lubich:prop} yields\begin{equation}
\begin{aligned}
\dot U&=(I-UU^T)(AU+QUR^{-1})\\\dot R&=U^TAUR+RU^TA^TU+U^TQU-RU^TC^TN^{-1}CUR ,
\end{aligned}\end{equation}    termed low-rank Riccati  equation in   \cite{bris2012low}. This may be used as a low-rank proxy to the Kalman filter. The obtained filter then performs inference  in the low-dimensional subspace spanned by the dominant eigenvectors \cite{le2022low}.

To approximate \eqref{Riccati} in $\Sii$ instead,  we replace $P$ with $Y=URU^T+s(I-UU^T)$ in the right-hand side of \eqref{Riccati} and we project the obtained matrix using \eqref{matrices:PPCA}.  
\begin{prop}[PPCA-Riccati]\label{PPCA:ric:prop}
     The lift on $\Sii$ of the orthogonal projection of the vector field defined by  \eqref{Riccati} writes
\begin{align}
\dot s&= \frac{1}{d-p}\tr\bigl((I-UU^T)(2sA+Q-s^2C^TN^{-1}C)\bigr)\label{Ric1}\\
\dot U&= (I-UU^T)(AUR+QU+sA^TU -sC^TN^{-1}CUR)( R-sI)^{-1}\label{Ric2}\\
\dot R&= U^TAUR+RU^TA^TU+U^TQU-RU^TC^TN^{-1}CUR. \label{Ric3} 
\end{align}\end{prop}
We have the following interpretation: \eqref{Ric3} is the original Riccati equation \eqref{Riccati} where matrices are projected onto the subspace encoded by $U$, \eqref{Ric2} resembles the Oja flow~\cite{Oja} that tracks the dominant subspace of a symmetric matrix, and \eqref{Ric1} provides an adaptation of parameter $s$ to reflect the inflation of the covariance under the vector field in the subspace orthogonal to $\mathrm{span}(U)$. Note that, \eqref{Ric1}  offers guarantees that $s(t)$ remains positive at all times, since as long as $(I-UU^T)Q\neq 0$, we have $s=0\Rightarrow \dot s >0$. Thus the lifted equations \eqref{Ric1}-\eqref{Ric3} preserve  our PPCA form at all times.

In terms of implementation, we will need a numerical integration method for the obtained differential equations. With  $h$ the time step and   $\dot s(t)=\delta s,\dot U(t)=\delta U,\dot R(t)=\delta R$, one needs to define $s(t+h),U(t+h),R(t+h)$ so as $Y(t+h)\in\Sii$. This can be performed using a retraction, see \cite{Absil}. We let
\begin{align}
&s(t+h)\leftarrow s(t)+h\delta s \\
&U(t+h)\leftarrow \mathrm{retraction\_qr} (U(t),h\delta U)  \label{st:ret} \\
&R(t+h)\leftarrow R^{1/2}\exp(hR^{-1/2}\delta R R^{-1/2})R^{1/2},\label{exp:ret}
\end{align}
where \eqref{st:ret} is based on the retraction ``$\mathrm{retraction\_qr} $" of the Stiefel manifold from the  Manopt toolbox \cite{Manopt}, and \eqref{exp:ret} is a natural retraction on PD matrices \cite{Bonnabel}, and where ``$\exp$" denotes the usual matrix exponential, whose computation is tractable in the small dimension $p$.  $\Si$ is similarly treated.  Note that  alternative schemes have been proposed on $\Sy$, see \cite{Lubich, Rouchon, nonnenmacher2008dynamical}.   \cite{lubich2014projector} even proposes a numerical integration method  that avoids inverting the factor $R$, {see \cite{ceruti2020time} for recent developments.}

Regarding the FA approach, we can approximate \eqref{Riccati} in $\Si$ as follows. We write $H:=\dot P=AP+PA^T+Q-PSP$ with $S=C^TN^{-1}C$ and $P=URU^T+\psi$. After some calculations this yields\begin{prop} [FA-Riccati]
     The lift on $\Si$ of the orthogonal projection of the vector field defined by  \eqref{Riccati} writes
\begin{equation*} 
\begin{aligned}
&{\rm diag}(\dot \psi)=(I-2\bar D+\Upsilon \Upsilon^T)^{+}~{{\rm diag}}{((I-UU^T)M(I-UU^T))}\\
&\dot U=(I-UU^T)(M- \delta \psi)UR^{-1}+(I-UU^T)(AU-\psi S U)\\
&\dot R=U^T (M- \delta \psi) U-U^T\psi S U R-R U^T S \psi U-R U^T SUR +U^TAUR+RU^TA^TU, 
\end{aligned}
\end{equation*}
where $M=A \psi+\psi A^T+Q-\psi S\psi$, $S=C^TN^{-1}C$ and $\bar D,\Upsilon$ as in \eqref{min1}.    \end{prop}The pseudo-inversion may be   avoided when $p\ll d$  using the Woodbury formula as in equation \eqref{argmin2}, leading to an update having a memory cost linear in $d$ if the order of operations is taken carefully.

 We now propose  two applications for the equations just obtained, one in the field of computational statistics where one seeks to approximate a fixed posterior distribution, for which we provide a tractable approximation in high dimension through a Riccati flow, and one in the field of Kalman filtering, with applications to robotics. The former is essentially an illustration of the applicability of the method, while the latter aims at comparing the various approaches through numerical simulations. 

\subsection{Application to computational statistics}\label{Optim:sec}
We seek to approximate a distribution in high dimension which is the stationary solution  of a Langevin equation. In this context, Markov-Chain Monte Carlo (MCMC) has been extensively studied in statistical physics and machine learning \cite{Andrieu2003}. Recently, variational inference \cite{lambertNIPS} has been used in this framework as an alternative form to MCMC to approximate the stationary distribution with a Gaussian distribution. The final approximation is defined as the asymptotic limit of a flow over the set of Gaussian distributions, parameterized by their mean and  covariance matrix. The flow for the covariance matrix has the form of a Riccati equation.   To illustrate the applicability of the methods developed in the present paper, we now show that the PPCA projection can be directly applied to these flows to provide a memory-efficient and computationally tractable algorithm in high dimension. 

\subsubsection{Wasserstein gradient flow for Gaussian variational inference}
In Bayesian statistics, one is often faced with the problem of approximating a target posterior distribution $\pi$ that is known up to a normalizing constant, that is, we have access to $ \nabla V(x)$ where $\pi(x)\propto \exp(-V(x)/\epsilon)$. To compute a Gaussian approximation of $\pi$ we may proceed as follows. The following stochastic differential equation (Langevin dynamics): 
  \begin{align}
 dx_t=-\nabla V(x_t)dt+\sqrt{2\varepsilon} dB_t,
  \end{align}
  is such that, under suitable assumptions, the marginal distribution of $x_t$ satisfies    the Fokker-Planck (partial differential) equation:
  \begin{align}
\frac{\partial p_t}{\partial t}= {\rm div}(Vp_t)+ \varepsilon \Delta p_t,
  \end{align}
  which has stationary distribution $\pi$. Thus, if one can approximate the solution to the latter equation by a Gaussian density, one may hope the latter tends asymptotically to a Gaussian approximation of $\pi.$
In the recent work \cite{lambertNIPS}, the solution of this PDE has been approximated with a Gaussian $p_t \approx q_t=\mathcal{N}(\mu_t,P_t)$ using variational inference. The following Gaussian flow: 
  \begin{align}
    \begin{aligned}
 &\dot \mu_t= -\mathbb{E}_{X\sim \mathcal{N}(\mu_t,P_t)}[\nabla V(X)],\\ 
 &\dot P_t=A_tP_t+P_tA^T_t+Q,\\
 & Q:=2\varepsilon I; \quad A_t:=-\mathbb{E}_{X \sim\mathcal{N}(\mu_t,P_t)}[\nabla^2 V(X)], 
     \end{aligned}  \label{covFlow}
 \end{align}
is shown to converge with   exponential rate  to the Gaussian distribution being the closest to the target distribution $\pi$ (in the sense of KL divergence), if $V$ is strongly convex, see  \cite[Appendix A] {lambertNIPS}. For large-scale problems, the differential equation on $P_t$ in equation \eqref{covFlow} is   problematic   because it requires storing at each time a $d \times d$ matrix. Moreover, the computation cost is (roughly) cubic in the dimension due to matrix products, even if we overlook the difficulty of computing the required expectations. This motivates a low-rank plus diagonal approximation to this Riccati-like equation. We readily see the   Wasserstein gradient flow \eqref{covFlow} lends itself to the present low-rank approximation framework, as it is of the form \eqref{Riccati}, provided that the expectations which define  $A_t$ may be computed, and that they can be written in a factorized form compatible with the high dimension.  

 \subsubsection{Particular case of a Gaussian target}

To fix ideas, let us see what the equations boil down to in the case where the target is Gaussian, that is,  $\pi \sim \mathcal{N}(m,\epsilon M )$, to be consistent with $\pi\propto \exp(-V/\epsilon)$ so that $V(x)=\frac{1}{2}(x-m)^TM^{-1}(x-m)$. In this case, the expectations can be computed analytically and the variational Gaussian flow  \eqref{covFlow} writes:
 \begin{align*}
& \dot \mu_t=-M^{-1}(\mu_t-m)\\
&\dot P_t=2\varepsilon I-M^{-1}P_t-P_tM^{-1}.
 \end{align*}
 The equations have as stationary point $(\mu,P)=(m,\epsilon M)$ which are the parameters of $\pi$, as expected. 

\begin{rem}
We see 
the ``pure" low-rank projection, defined by equation \eqref{Delta}, of the covariance $P_t$ writes:
\[\dot  U=-(I-UU^T)M^{-1}U,\]
and we recover minus the Oja flow \cite{Oja}, that is, the flow in a form that tracks the   eigenvectors corresponding to the smallest eigenvalues of $M^{-1}$, thus, the dominant eigenspace of $M$.     
\end{rem}

\subsubsection{PPCA approximation in the large-scale case}

We now come back to a general target posterior $\pi\propto\exp(-V/\epsilon)$ and seek to approximate the Gaussian flow \eqref{covFlow} in a tractable form with high $d$. The first step is to approximate  the expectation  under the Gaussian distribution appearing in the Gaussian flow \eqref{covFlow} with $K$  Monte-Carlo  samples. To this aim, we may let appear an outer product of the form:
   \begin{align*}
 &\mathbb{E}_{  \mathcal{N}(\mu_t,P_t)}[\nabla^2 V(X)]P_t=\mathbb{E}_{  \mathcal{N}(\mu_t,P_t)}[\nabla V(X) (X-\mu_t)^T] \label{Exp}\\
 &\approx \frac{1}{K} \sum_{k=1}^K \nabla V(x^k_t) (x^k_t-\mu_t)^T = D_tB_t^T \quad \text{where} \quad x^k_t \sim q_t,\\
 &D_t:=\frac{1}{\sqrt{K}}  \left(  \nabla V(x^1_t), \dots , \nabla V(x^N_t)  \right) \quad B_t:=\frac{1}{\sqrt{K}}   \left(x^1_t-\mu_t, \dots , x^N_t-\mu_t  \right), 
 \end{align*}
where the first equality comes from integration by parts (Stein Lemma \cite{Lin2019}). We now only need to store the matrices $B_t$ and $D_t$ of size $d \times K$. In practice, the sampling is done after the discretization of these ODEs at each integration step. 

As $D_t$ and $B_t$ depend  on the samples $x_t^k$ generated from the distribution $q_t \sim \mathcal{N}(\mu_t,P_t)$, where in our PPCA approximation we let $P_t= URU^T+s(I-UU^T) $, we need to be able to sample from a Gaussian distribution having such a factorized covariance matrix. A simple way to do so is as follows. We first note that  if we take $Z_1\sim\mathcal N(0,URU^T)$ and $Z_1\sim\mathcal N(0,s(I-UU^T))$ independent, then $Z_1+Z_2$ has the desired covariance. To sample $Z_1$, one may sample a small-dimensional variable $Z_3\sim \mathcal N(0,R)$ and let $Z_1=UZ_3$. To sample $Z_2$, we may sample $Z_4\sim \mathcal N(0,s I)$ and let $Z_2=(I-UU^T)Z_4.$

 The ODE we seek to approximate  now takes the form:
 \begin{align}
 &\dot{P_t}=2\varepsilon I-(D_t B_t^T + B_t D_t^T).
 \end{align}
 
We   project this equation onto the PPCA subset. This is achieved thanks to Proposition \ref{PPCA:ric:prop}, and it yields
\begin{align*}
\dot s&= \frac{1}{d-p}\tr\bigl(2(I-UU^T)(\varepsilon I-D_t B_t^T )\bigr),\\
\dot U&= (I-UU^T)(2\varepsilon I-D_t B_t^T - B_t D_t^T)U( R-sI)^{-1},\\
\dot R&= U^T(2\varepsilon I-D_t B_t^T - B_t D_t^T)U.  
\end{align*}

These operations can be computed in a memory-efficient way using the relation $\tr D_tB_t^T=\sum_{i=1}^d D_t[i,:] B_t[:,i]$ and $U^T D_tB_t^T U=(U^TD_t) (U^TB_t)^T$  to only manipulate vectors or matrices of size $p \times K$. 

We have briefly shown that using the results of the present paper, the recent Wasserstein gradient flow of \cite{lambertNIPS} for variational inference is amenable to a PPCA approximation being compatible with the high dimension. We now turn to another application and compare numerically  the various low-rank approaches.

\subsection{Application to Kalman filtering} \label{appliKalman}

In this section, we consider the more standard problem of the Kalman filter in continuous time known as the Kalman-Bucy filter. In  Kalman filtering, we seek to estimate hidden physical quantities that evolve over time and which are partially observed  through a linear model. The covariance of the state $P$ satisfies the Riccati equation \eqref{Riccati}.

To assess the various low-rank approximations, we consider an example inspired by robotics. We consider a swarm of $d/2$ agents in the 2D plane, each equipped with  motion sensors, and governed by the following dynamics
\begin{align}
\forall~ 1\leq i\leq  {d}/{2}\qquad \dotex  X_i=u_i(t) +w_i(t),    
\end{align} where $X_i\in\RR^2$ denotes the position of agent $i$,   $w_i(t)$ is a white noise, and $u_i(t)$   a control input.  
Measurements consist of relative position between some agents being neighbors in a visibility graph, corrupted by noise $v_{ij}(t)$, i.e.,
\begin{align}
    y_{ij}(t)=X_j(t)-X_i(t)+v_{ij}(t),\quad {(i,j)\in\Gamma}.    
\end{align}

Moreover, there is a ``queen", say, agent 1, having relatively more computational capacity onboard--albeit limited--and which is equipped with a GPS, that is, we also measure $y_1(t)=X_1(t)+v_1(t)$.  The queen receives the dynamical motions $u_i(t)$ as well as the measurements $y_{ij}(t)$, and estimates the state of the whole swarm onboard, through a low-rank filter compatible with its modest computational capabilities.

We integrate the Riccati equation  with an Euler scheme (with step $0.01s$) during $10s$ and consider an initial covariance matrix  factorized as  purely low-rank $P_0=URU^T$ with $R=2 {I}_p$ and $U$ a  random Stiefel matrix, to allow for a common starting point. We set noise covariance to be $N=2I$ and   process noise covariance  $Q=D$ where $D$ is a  diagonal matrix whose diagonal consists of positive values 
 dispersed around 1, reflecting discrepancies in the accuracy of motion sensors of the agents.  In our experiments, each agent sees one other  agent (that might be the queen), randomly picked at the beginning.  We   tested changing the visibility graph over time but that does not change the nature of the following results.

 We compare the low-rank approximation \cite{Lubich,Rouchon} with the two variants of our low-rank + diagonal approximation in dimension $d=200$, that is, 100 agents, and a latent dimension $p=8$ or $p=50$ in Figure \ref{figureXP1}. 
We clearly observe the  proposed approximations outperform the former in terms of distance to the true covariance matrix. Moreover, the order of the curves is as expected: Projections onto an increasing sequence of submanifolds yield 
 in turn increasing accuracy  for the matrix differential equation approximation.

\begin{figure}[!ht]
\includegraphics[scale=0.9]{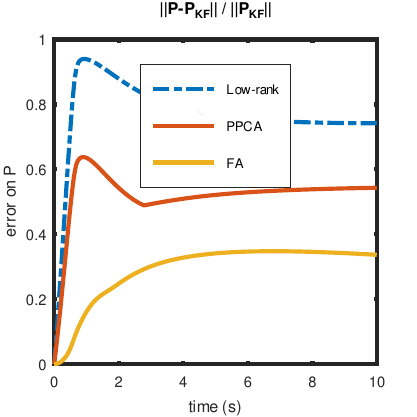}
\includegraphics[scale=0.9]{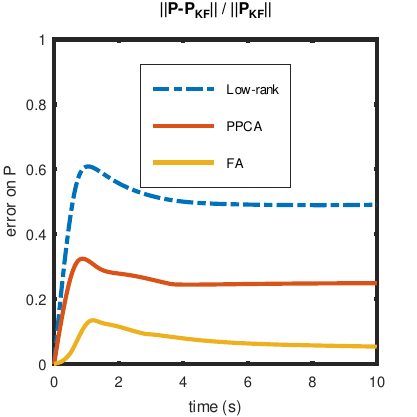}
\caption{$d=200$ with $p=8$ (left) and $p=50$ (right). Normalized distance between the covariance matrix computed from the true full-rank Riccati equation and the ones computed from the low-rank, low-rank + diagonal (FA) and low-rank + isotropic diagonal (PPCA) for the swarm  example.  } 
\label{figureXP1}
\end{figure}

 To assess the effect of approximating the Riccati equation on the Kalman filter's state estimates, we have compared the full-rank KF's optimal  estimates with  those obtained by its computationally cheaper variants, for  a randomly distributed initial error,  in a noise-free setting, that is, when using the Kalman filter as an observer, to get more legible curves.
The results are given in Figure \ref{figureXP2}. We see the deviation to the optimal estimates is much more contained when adding a diagonal matrix.

\begin{figure}[!ht]
\includegraphics[scale=.9]{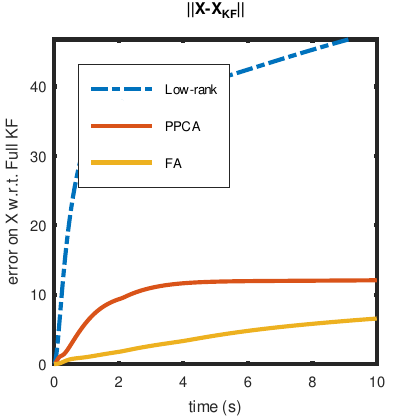}
\includegraphics[scale=.9]{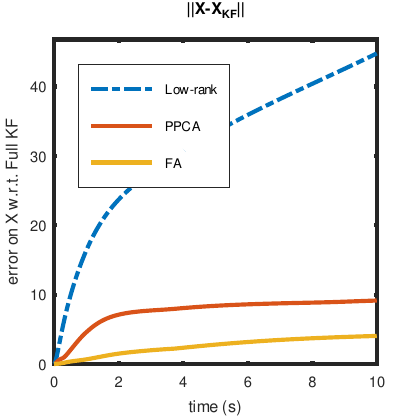}
\caption{$d=200$ with $p=8$ (left) and $p=50$ (right). Norm of the error over time between the filters' state estimates and the full KF's optimal estimate, for a randomly picked initial error in a noise-free setting. } 
\label{figureXP2}
\end{figure}
It  is striking to see the PPCA method with $p=8$ competes with the low-rank method of \cite{Rouchon} (arxiv preprint version) with $p=50$  in terms of covariance matrix approximation, at the expense of only 1 more scalar parameter. We also observe  PPCA with $p=8$ outperforms low-rank with $p=50$ in terms of state estimate accuracy. More generally, the experiments show that adding a diagonal term provides an efficient alternative to increasing the rank.

 Two remarks are in order. First, we note that, although the curves indicate a clear improvement of FA over PPCA, this is in fact largely due to the choice of anisotropic process noise, reflecting a  
 discrepancy in the motion sensors' accuracy. When making the problem more ``isotropic", the differences between the two   diminish. 
Then, we  also  observed  that the initial condition $P(0)$ plays an important role. Although we took $P(0)\in\Sy$ to allow for a common starting point, we noted that departing from $P(0)\in\Sii$ naturally reduces  the gap between FA and  PPCA.

\subsubsection{Source code}
 {The code is available on Github for Octave or Matlab at } \url{https://github.com/marc-h-lambert/Riccati-PPCA}. It provides  projections onto   $\mathcal T\Sy$, $\mathcal T\Sii$ and $\mathcal T\Si$ of matrices of the form $H=GG^T$, or diagonal, or a linear combination, and scripts to redo the numerical   experiments.

\subsubsection{Interest of invertible approximations in Kalman filtering}
We conclude this section with a discussion about the relevance of low-rank approximations in Kalman filtering.   To illustrate the shortcomings of purely low-rank based methods, e.g.,    \cite{Rouchon,Pham,bonnabel2013geometry,Lubich,le2022low} for Kalman filtering, we consider a tutorial case, see Example 6.2.10 of \cite{oksendal2013stochastic}. 
Consider  noisy observations of a Brownian motion  in a large dimension 
$$dX=dw,\quad dY=dX+dv,\quad P(t_0)=\bar P,$$with $w,v$ Wiener process noises with  covariance matrices $Q=\lambda I,~N=\nu I$.  Assume  one wants to filter the noise  out.

\paragraph{Low-rank KF}The low-rank Riccati equation of \cite{Bonnabel,Rouchon,le2022low} writes
$$\dot U=0,\quad \dot R=U^TQU-RU^TN^{-1}UR$$as we have $\dot U=(I-UU^T)QUR^{-1}=\lambda (I-UU^T)UR^{-1}$. As a result, $R$ stabilizes to the steady-state solution $\sqrt{ {\lambda\nu}}I_p $. Let us see how the corresponding steady-state Kalman filter updates the state. The static Kalman gain writes $K=P_\infty C^TN^{-1}=\sqrt{{\lambda}/{\nu}}~U_0U_0^T$ where $U_0$ is the initial value of $U$. The steady-state KF equations are 
\begin{align*}
    d\hat X&=K(dY-\hat X dt)=\sqrt{ {\lambda/\nu}}~UU^T  (dY- \hat X dt).
\end{align*}
We see that $P$ is a degenerate low-rank matrix and the innovation vector $dY- \hat X dt$ is always projected onto the same subspace by the operator $UU^T$.  

\paragraph{Low-rank plus diagonal KF in PPCA form}By contrast,  \eqref{Ric1}-\eqref{Ric3} write:
$$\dot U=0, \quad \dot R=\lambda I_p-\frac{1}{\nu} R^2, \quad \dot s= \lambda -\frac{s^2}{\nu}.$$
The steady-state solution gives $R=\sqrt{\lambda \nu} I_p$ and $s=\sqrt{\lambda \nu}$.
The associated KF equations write now:
\begin{align*}
    P&=URU^T+s(I-UU^T)=\sqrt{ {\lambda\nu}}I \\
    d\hat X&=K(dY-\hat X dt)=\sqrt{ {\lambda/ \nu}} ~ (dY- \hat X dt).
\end{align*}
In this particular example, we see we recover  the equations of the full-rank steady-state Kalman filter. 

 \begin{prop}
 Using the low-rank KF, the estimation error incurred grows unbounded  $E(||X-\hat X||^2)\to +\infty$. By contrast,   using our low-rank plus diagonal approximation,  $E(||X-\hat X||^2)\to \sqrt{ {\lambda/ \nu}}d$. Hence, the error is here optimally contained. 
 \end{prop}
 \begin{proof}
     For the low-rank KF we have $d[(I-UU^T)(X-\hat X)]=(I-UU^T)dw$. Thus $\dotex E(||X-\hat X||^2)\geq (d-p)\lambda$.  By contrast, for the proposed KF, the estimation error is an Ornstein-Uhlenbeck process whose norm stabilizes at $\sqrt{ {\lambda/ \nu}}d$.
 \end{proof}Contrary to the   robotics example, the latter problem obviously lends itself to the proposed approximation, owing to $C,Q,N$ being diagonal and isotropic, leading to a diagonal and isotropic (stationary) covariance that is fully recovered on $\Sii$. However, it clearly illustrates the interest of maintaining an invertible covariance matrix in Kalman filtering.   
 \section{Conclusion}

We have derived the orthogonal projection of any tangent vector to the set of PSD matrices onto the tangent space to low-rank plus diagonal PSD matrices. This allows for an attractive  alternative to previous low-rank approximations of differential equations defined on the set of PSD matrices, in that it retains their computational efficiency while allowing for more flexible, full-rank, approximations, leading to better results for large-scale filtering.   Our conclusion is as follows:
\begin{enumerate}
    \item Existing (pure) low-rank PSD approximations have the shortcoming of having many null eigenvalues, which comes with consequences. If they model a Gaussian covariance matrix, we cannot compute the precision matrix, and there is no density. In Kalman filtering, this results in overconfidence, 
 potentially leading to large error drift in turn.
    \item The proposed PPCA approximation is about as efficient in terms of memory and computational cost while having none of those drawbacks. We recommend its systematic use. 
    \item The proposed FA  approximation additionally closely  captures the diagonal elements. However, it comes at a greater price, albeit retaining linearity in $d$. Its use shall be reserved for problems known to have a strongly non-isotropic diagonal, which  one wants to capture well. 
\end{enumerate}
{As a perspective, we see that while PPCA is computationally
efficient, it doesn't reach the error accuracies achieved by the FA format. It could prove useful to seek a
practical criterion for switching between  these two formats. 
 Besides, a criterion for choosing the rank $p$ (or adapting it dynamically over time) may also prove useful. A first step to address those perspectives is to observe that the approximation error $\norm{H-P(H)}^2$  can   be computed at all times with a cost being linear in $d$ too, allowing for approximation quality assessment, as shown in  Appendix \ref{AppendixB}. An other interesting perspective would be to explore how one could bring the proposed method  to bear on the dynamic factor analysis  problem, see e.g.,  \cite{falconi2023robust}.}
 
 \section*{Acknowledgments}
This work was funded by the French Defence procurement agency (DGA) and by the French government under the management of Agence Nationale de la Recherche as part of the “Investissements d’avenir” program, reference ANR-19-P3IA-0001(PRAIRIE 3IA Institute). We also acknowledge support from the European Research Council (grant SEQUOIA 724063).

We thank Pierre-Antoine Absil and Guillaume Olikier for enlightening discussions about the submanifold structure of $\Sy$. We also thank the anonymous reviewers for their suggestions. 
 
\bibliographystyle{plain}
\bibliography{CDC23}

\appendix
\section{Proof of Lemma \ref{lemme}} \label{AppendixA}
 
Letting $\delta \psi_1,\cdots,\delta \psi_d$ denote the diagonal elements of $\delta \psi$,  the cost \eqref{cost0} writes
\begin{align*}\bar C(\delta \psi)&=\sum_{i=1}^n(h_{ii}-\delta \psi_i)^2+4\tr(UU^TH\delta \psi)-2\tr(UU^T\delta \psi^2)\\&\hspace{1cm}-2\tr\bigl(U^THUU^T\delta \psi U\bigr)+\tr\bigl(U^T\delta \psi UU^T\delta \psi U\bigr)\\ &:=\textcircled{1}+\textcircled{2}+\textcircled{3}+\textcircled{4}+\textcircled{5}.\end{align*}
Let us analyze each term, and write them as functions of  $\overline{\delta \psi}:=(\delta \psi_1,\cdots,\delta \psi_d)^T={\rm diag}(\delta \psi)$.   
\begin{align*}
\textcircled{1}&=\sum_{i=1}^n(h_{ii}-\delta \psi_i)^2=(\bar h-\overline{\delta \psi})^T(\bar h-\overline{\delta \psi}), \quad \bar h:={\rm diag}(H)
\\
\textcircled{2}&=4\tr(UU^TH\delta \psi) =4\bar h_U^T\overline{\delta \psi}, \quad \bar h_U:={\rm diag}(UU^TH)\\
\textcircled{3}&=-2\tr(UU^T\delta \psi^2)=-2\sum_{k=1}^d\delta \psi_k^2(\sum_{j=1}^rU_{kj}^2)=-2\overline{\delta \psi}^T\bar D\overline{\delta \psi}
\end{align*}
where we let $\bar D$ be the diagonal matrix defined by $\bar D_{ii}=\sum_{j=1}^rU_{ij}^2$. As concerns the fourth term we have\begin{align*}\textcircled{4}&= -2\tr\bigl(U^THUU^T\delta \psi U\bigr) \\&=-2 \sum_{1\leq i,j\leq r} (U^THU)_{ij}(\sum_{k=1}^d\delta \psi_kU_{ki}U_{kj})\\&=-2\sum_{k=1}^d\delta \psi_k \sum_{1\leq i,j\leq r} (U^THU)_{ij} U_{ki}U_{kj}\\&=-2\Lambda^T\overline{\delta \psi},\quad 
 \Lambda=(\Lambda_1,\cdots,\Lambda_d)^T\in\RR^d\end{align*}
with  $\Lambda_k=\sum_{1\leq i,j\leq r} (U^THU)_{ij} U_{ki}U_{kj}$ for each $k$. Finally
\begin{align*} 
\textcircled{5}&=  \tr\bigl(U^T\delta \psi UU^T\delta \psi U\bigr)=  \sum_{1\leq i,j\leq r} (\sum_{k=1}^d\delta \psi_kU_{ki}U_{kj})^2 \\
&= \sum_{1\leq i,j\leq r}(\Gamma_{ij}^T\overline{\delta \psi})^2,
\end{align*}
where each $\Gamma_{ij}\in\RR^d$ is a vector whose $k$-th component is $U_{ki}U_{kj}$. Thus we seek to minimize
\begin{align*} 
&\bar C(\overline{\delta \psi})=(\bar h-\overline{\delta \psi})^T(\bar h-\overline{\delta \psi}) 
 +4\bar h_U^T\overline{\delta \psi}-2\overline{\delta \psi}^T \bar D\overline{\delta \psi}- 2\Lambda^T\overline{\delta \psi}\\
 &\hspace{6 cm}+\sum_{1\leq i,j\leq r}(\Gamma_{ij}^T\overline{\delta \psi})^2.
 \end{align*}
 Thus
\begin{align*}
& \nabla_{\overline{\delta \psi}}\bar  C(\overline{\delta \psi})=2(\overline{\delta \psi}-\bar h)+ 4\bar h_U-4\bar D\overline{\delta \psi} -2\Lambda+2\sum_{1\leq i,j\leq r}\Gamma_{ij}\Gamma_{ij}^T\overline{\delta \psi}\\
&\Rightarrow \frac{1}{2}\nabla_{\overline{\delta \psi}}\bar  C(\overline{\delta \psi})=\bigl(I-2\bar D+\sum_{1\leq i,j\leq r}\Gamma_{ij}\Gamma_{ij}^T\bigr)\overline{\delta \psi}-\Lambda+2\bar h_U-\bar h.
\end{align*}
 And the optimizer is given by
 $$
 \overline{\delta \psi}^*=\bigl(I-2\bar D+\sum_{1\leq i,j\leq r}\Gamma_{ij}\Gamma_{ij}^T\bigr)^{-1}\bigl(\bar h-2\bar h_U+\Lambda\bigr).
 $$
Finally, we note that $\sum_{1\leq i,j\leq r}\Gamma_{ij}\Gamma_{ij}^T=\Upsilon \Upsilon^T$ where $\Upsilon$ is a $d\times p^2$ matrix whose columns consist of  the $p^2$ vectors $ \Gamma_{ij}$.  To avoid double counting we may take for $\Upsilon$  a matrix of size $d\times \frac{p(p+1)}{2}$ whose first  $\frac{p(p-1)}{2}$ columns are given by the vectors $\sqrt{2}\Gamma_{ij}$, $i<j$, and whose last $r$ columns are the  $\Gamma_{ii}$'s.

\section{Monitoring the quality of approximation}\label{AppendixB}
 
An advantage  of the method is that one may monitor the quality of the approximation by computing the discrepancy $\norm{H-P(H)}^2$ at all times, with a cost being linear in $d$. This opens up for methods that could dynamically adapt the parameter $p$, or switch between PPCA and FA. Indeed, have the following result.

\begin{prop}In the PPCA case, we have$$  \norm{H-P_{U,R,s}(H)}^2=\tr\bigl ((I-UU^T) [H-\delta s I]^T [H-\delta s I] \bigr ).$$This may be compared to $\norm{H}^2= \norm{H-P(H)}^2+\norm{P(H)}^2$  to assess the quality of the approximation.\end{prop} 
As for the FA case, we have advantageously rewritten $\norm{H-P_{U,R,\psi}(H)}^2$ as  the vector cost \eqref{barcost}, allowing for the following result.

\begin{prop}In the FA case, $  \norm{H-P_{U,R, \psi}(H)}^2$ may be computed as  \eqref{barcost} letting $$\alpha= \tr\bigl(  H^2\bigr)-2\tr\bigl(UU^T  H^2\bigr)+\tr\bigl(UU^T  H UU^T H\bigr)-\bar h^T\bar h.$$This may be compared to $\norm{H}^2= \norm{H-P(H)}^2+\norm{P(H)}^2$  to assess the quality of the approximation.\end{prop}
\begin{proof}To   compute the constant $\alpha$ in   \eqref{barcost},  we write $\tilde C(0)=\alpha+\bar h^T\bar h$. This is also equal to 
 \eqref{cost0} at $\delta \psi=0$, i.e., letting $\tilde H=H$ in the expression. This yields $\alpha$.  
\end{proof}

\end{document}